\documentclass[10pt,conference]{IEEEtran}
\usepackage{amsmath,graphicx,cite,bm,amssymb,amsthm,enumerate,epsfig,psfrag,cases,mathtools}

\usepackage{geometry}
\renewcommand{\(}{\left(}
\renewcommand{\)}{\right)}
\renewcommand{\[}{\left[}
\renewcommand{\]}{\right]}

\renewcommand{\S}{\mathbf{S}}

\newcommand{\W}{\mathbf{W}}

\newcommand{\x}{\mathbf{x}}

\newcommand{\C}{\mathbf{C}}

\newcommand{\A}{\mathbf{A}}

\newcommand{\M}{\mathbf{M}}
\newcommand{\N}{\mathbb{N}}

\newcommand{\X}{\mathbf{X}}
\newcommand{\Y}{\mathbf{Y}}

\newcommand{\Tr}[1]{{\rm{Tr}}\left(#1\right)}

\newcommand{\End}[1]{{\rm{End}}}

\renewcommand{\log}[1]{{\rm{log}}#1}

\newtheorem{lemma}{Lemma}

\newtheorem{definition}{Definition}

\newtheorem{prop}{Proposition}

\newcommand{\norm}[1]{\left\lVert#1\right\rVert}

\usepackage[ruled,vlined,resetcount]{algorithm2e}
\usepackage{algorithmic,subfig}
\usepackage{fontenc}
\usepackage{inputenc}
\usepackage[square,sort,compress,comma,numbers]{natbib}
\usepackage{epstopdf,pst-node}
\usepackage{mathtools}
\usepackage{flushend}

\newcommand{\mypm}{\mathbin{\smash{%
\raisebox{0.35ex}{%
            $\underset{\raisebox{0.5ex}{$\smash -$}}{\smash+}$%
            }%
        }%
    }%
}

\DeclarePairedDelimiter\floor{\lfloor}{\rfloor}

\begin{document}
\title{On the Spectral Norms of \\ Pseudo-Wigner and Related Matrices}

\author{Ilya Soloveychik and Vahid Tarokh, \\ John A. Paulson School of Engineering and Applied Sciences, Harvard University}
\maketitle

\begin{abstract}
We investigate the spectral norms of symmetric $N \times N$ matrices from two pseudo-random ensembles. The first is the pseudo-Wigner ensemble introduced in ``Pseudo-Wigner Matrices'' by Soloveychik, Xiang and Tarokh and the second is its Sample Covariance-type analog defined in this work. Both ensembles are defined through the concept of $r$-independence by controlling the amount of randomness in the underlying matrices, and can be constructed from dual BCH codes. We show that when the measure of randomness $r$ grows as $N^\rho$, where $\rho \in (0,1]$ and $\varepsilon > 0$, the norm of the matrices is almost surely within $o\(\frac{\log^{1+\varepsilon} N}{N^{\min[\rho,2/3]}}\)$ distance from $1$. Numerical simulations verifying the obtained results are provided.
\end{abstract}


\section{Introduction}
Random matrices have been a very active area of research for the last few decades and found enormous applications in various areas of modern mathematics, physics, engineering, biological modeling, and other fields \cite{akemann2011oxford}. In this article, we focus on two types of square symmetric matrices: 1) sign ($\mypm 1$) matrices and 2) Sample Covariance Matrices (SCM) of sign vectors. 

Random square symmetric sign matrices were originally examined by Wigner \cite{wigner1955characteristic}. He proved that if the elements of the upper triangle of an $N \times N$ symmetric matrix (including the main diagonal) are independent Rademacher ($\mypm 1$ with equal probabilities) random variables, then as $N \to \infty$ a properly scaled empirical spectral measure converges to the semicircular law. Wigner originally showed convergence in expectation, which was later improved to convergence in probability \cite{grenander2008probabilities} and to almost sure weak convergence \cite{arnold1967asymptotic}. The spectral behavior of SCMs formed from $p$ independent $N$ dimensional vectors with independent entries and $\frac{p}{N}\to \gamma \in (0,1)$, was for the first time rigorously investigated by Marchenko and Pastur \cite{marchenko1967distribution}. They showed that (actually, under weaker conditions on dependencies among vector entries) the limiting spectrum converges to a non-random law.

In many engineering applications, one needs to simulate random matrices. The most natural way to generate an instance of a random $N \times N$ sign matrix is to toss a fair coin $\frac{N(N+1)}{2}$ times, fill the upper triangular part of a matrix with the outcomes and reflect the upper triangular part into the lower. Similarly, to get a random SCM matrix one would need to toss a coin $pN \approx \gamma N^2$ times. Unfortunately, for large $N$ such approach would require a powerful source of randomness due to the independence condition \cite{gentle2013random}. In addition, when the data is generated by a truly random source, atypical  \textit{non-random looking} outcomes have non-zero probability of showing up. Yet another issue is that any experiment involving tossing a coin would be impossible to reproduce. All these reasons stimulated researchers and engineers from different areas to seek approaches of generating \textit{random-looking} data usually referred to as \textit{pseudo-random} sources or sequences of binary digits \cite{zepernick2013pseudo, golomb1982shift}. A wide spectrum of pseudo-random number generating algorithms have found applications in a large variety of fields including radar, digital signal processing, CDMA, error correction, cryptographic systems, and Monte Carlo simulations, navigation systems, scrambling, coding theory, etc. \cite{zepernick2013pseudo}.

The term \textit{pseudo-random} is used to emphasize that the binary data at hand is indeed generated by an entirely deterministic causal process with low algorithmic complexity, but its statistical properties resemble some of the properties of data generated by tossing a fair coin. Remarkably, most efforts were focused on one dimensional pseudo-random sequences \cite{zepernick2013pseudo, golomb1982shift} due to their natural applications and to the relative simplicity of their analytical treatment. The study of pseudo-random arrays and matrices was launched around the same time \cite{reed1962note, macwilliams1976pseudo, imai1977theory, sakata1981determining}. Among the known two dimensional pseudo-random constructions the most popular are the so-called perfect maps \cite{reed1962note, paterson1994perfect, etzion1988constructions} and two dimensional cyclic codes \cite{imai1977theory, sakata1981determining}. However, none of these works considered spectral properties as the defining statistical features for their constructions.

Specific pseudo-random constructions usually develop from a set of properties mimicking truly random data, and attempt to come up with deterministic ways of reproducing these properties. Following this approach, in \cite{soloveychik2017pseudo} we proposed a framework allowing construction of symmetric sign matrices of low Kolmogorov complexity with spectra converging to the semicircular law. Here we extend the ideas of \cite{soloveychik2017pseudo} to the construction of low complexity SCMs with spectra converging to Marchenko-Pastur law. In a related work \cite{babadi2011spectral}, the authors show that if the columns of matrices are randomly chosen from a properly designed binary code, their spectra converge to Marchenko-Pastur law as is the case for our construction. However all these works do not examine finer characteristics of the proposed matrices. In the present article we go beyond limiting spectral measures. We require more moments of the pseudo-random construction to match those of the truly random ensembles which enables us to capture the behavior of the extreme eigenvalues (spectral norm). As a tradeoff, we pay a penalty for that by increased Kolmogorov complexity. We also provide an explicit construction of both Wigner-type and SCM-type ensembles from dual BCH codes and support our theoretical results by numerical simulations.

The outline of this paper is given next. Section \ref{sec:pwe} provides the original truly random ensembles and their properties. In Section \ref{sec:pseudo-rand}, we introduce our pseudo-random ensembles through the concept of $r$-independence and demonstrate that by increasing the amount of randomness involved in their construction, we can mimic finer properties of the true random matrices. The main results about the spectral norms are presented in Section \ref{sec:spect_norm} followed by numerical tests in Section \ref{sec:num}.

\textbf{Notation.} For a real $x,\; \floor*{x}$ stands for the largest integer not exceeding $x$. For a real random variable $X$, we write $F_X(x)$ for its cumulative distribution function (c.d.f.) and $f_X(x)$ for its probability density function (p.d.f.). For two real functions $g(x)$ and $f(x)$ of a real or natural argument, we say $g(x)=o(f(x))$ if $\lim\limits_{x \to +\infty} g/f=0$ and $g(x)=O(f(x))$ if $\lim\limits_{x \to +\infty} g/f < +\infty$.

\section{Random Matrix Ensembles} 
\label{sec:pwe}
Denote the spectrum of a symmetric real matrix $\S_N$ by
\begin{equation}
\lambda_1(\S_N) \leqslant \dots \leqslant \lambda_N(\S_N),
\end{equation}
the c.d.f. associated with it by
\begin{equation}
F_{\S_N}(x) = \frac{1}{N}\sum_{i=1}^N \theta(x-\lambda_i(\S_N)),
\end{equation}
where $\theta(x)$ is the unit step function at zero, and the spectral norm by
\begin{equation}
\norm{\S_N} = \max[|\lambda_1(\S_N)|,|\lambda_N(\S_N)|].
\end{equation}
The $l$-th empirical moment of $\S_N$ reads as
\begin{equation}
\label{eq:mom_def_trace}
\int x^l dF_{\S_N} = \frac{1}{N}\Tr{\S_N^l}.
\end{equation}

Next, we introduce two random ensembles. We will mimic their spectral properties by pseudo-random constructions in Section \ref{sec:pseudo-rand}.

\subsection{Wigner Matrices and the Semicircular Law}
The first rigorous study of a random matrix ensemble was performed by Wigner in his seminal work \cite{wigner1955characteristic, wigner1958distribution}. Wigner's ensemble $\mathcal{W}_N$ is the set $S_N$ of all $N \times N$ matrices with $\mypm \frac{1}{2\sqrt{N}}$ entries endowed with the uniform probability measure. 

Let $F_{W}$ be the c.d.f. of the standard semicircular law with the p.d.f.
\begin{equation}
\label{eq:W_pdf_def}
f_{W}(x) = \begin{cases} \frac{2}{\pi} \sqrt{1-x^2},& -1 \leqslant x \leqslant 1, \\ \qquad 0, & \text{otherwise}. \end{cases}
\end{equation}
The moments of this distribution read as
\begin{equation}
\label{eq:dc_mom}
\mu_{W}(s) = \int_{-\infty}^{+\infty} x^s d F_{W} = \begin{cases} \frac{1}{2^s}\C_{s/2}, & s \text{ even}, \\ \quad 0, & s \text{ odd},\end{cases}
\end{equation}
where
\begin{equation}
\C_{s/2} = \frac{s!}{\(\frac{s}{2}\)!\(\frac{s}{2}+1\)!}
\end{equation}
are Catalan numbers. Stirling's approximation yields
\begin{equation}
\frac{1}{2^s}\C_{s/2} = \sqrt{\frac{8}{\pi s^3}}(1 + o(1)),\quad s \to +\infty.
\end{equation}

Using the so-called method of moments, Wigner demonstrated \cite{wigner1955characteristic} that the empirical spectral measures of matrices from $\mathcal{W}_N$ converge in expectation to the semicircular law (\ref{eq:W_pdf_def}). In a follow up article he improved this result to convergence in probability \cite{wigner1958distribution}. Almost sure weak convergence \cite{arnold1967asymptotic} and other asymptotic results were obtained later \cite{anderson2010introduction}.

Here we focus on a series of results obtained by Soshnikov and Sinai \cite{sinai1998central, sinai1998refinement, soshnikov1999universality}. These papers developed a combinatorial technique enabling exact quantification of the high-order expected moments of Wigner matrices, and led to the proof of universality of the joint distribution of their largest eigenvalues.

\begin{lemma}[Corollary of Main Theorem from \cite{sinai1998central}]
\label{lem:mom_conv_norm_or}
Let $\W_N \in \mathcal{W}_N$ and $s_N = o\(N^{2/3}\)$, then
\begin{equation}
\mathbb{E}\[\Tr{\W_N^{s_N}}\] = \begin{cases} \sqrt{\frac{8}{\pi s_N^3}}N(1 + o(1)), & s_N \text{ even}, \\ 0, & s_N \text{ odd},\end{cases}
\end{equation}
as $N \to +\infty$, and the random variables
\begin{equation}
\Tr{\W_N^{s_N}} - \mathbb{E} \[\Tr{\W_N^{s_N}}\]
\end{equation}
converge in distribution to the normal law $\mathcal{N}\(0,\frac{1}{\pi}\)$.
\end{lemma}
This result in particular implies almost sure weak convergence of the empirical spectra of matrices from Wigner's ensemble to the semicircular law \cite{anderson2010introduction}. Below we also use the following variation of a result proven in \cite{soshnikov1999universality}.

\begin{lemma}[Corollary of Theorem 2 from \cite{soshnikov1999universality}]
Let $\W_N \in \mathcal{W}_{N}$, then for any sequence $s_N = O\(N^{2/3}\)$,
\begin{equation}
\mathbb{E}\[\Tr{\W_N^{s_N}}\] \leqslant c_W(s_N)N,
\end{equation}
where $c_W(s_N)$ is bounded uniformly over $N$.
\end{lemma}

\subsection{Sample Covariance Matrices and Marchenko-Pastur Law}
Let $\mathcal{X}_{N,p}$ be the set $M_{N,p}$ of $N \times p$ matrices with $\mypm \frac{1}{\sqrt{N}}$ entries endowed with the uniform probability measure. Below we consider a setting where the dimensions $N$ and $p_N = p(N)$ grow such that the limit
\begin{equation}
\gamma = \lim_{N \to \infty} \frac{p_N}{N},
\end{equation}
exists. The spectra of the SCMs $\X_N^\top\X_N$ with $\X_N \in \M_{N,p}$ are invariant under the replacement of $\X_N$ with $\X_N^\top$ up to zero eigenvalues, therefore, without loss of generality we assume $\gamma \leqslant 1$. The Marchenko-Pastur distribution is defined through its p.d.f. as
\begin{equation}
\label{eq:MP_pdf_def}
f_{MP}(x) = \begin{cases} \frac{1}{2\pi \gamma x}\sqrt{(b-x)(x-a)}, & a \leqslant x \leqslant b, \\ \qquad\qquad\quad 0, & \text{otherwise}, \end{cases}
\end{equation}
where
\begin{equation}
a = (1-\sqrt{\gamma})^2,\qquad b = (1+\sqrt{\gamma})^2.
\end{equation}
The moments of this distribution read as
\begin{equation}
\label{eq:dc_mom}
\mu_{W}(s) = \int_{-\infty}^{+\infty} x^s f_{W} dx = \sum_{k=1}^{s} \gamma^k \bm{N}(s_N,k),
\end{equation}
where
\begin{equation}
\bm{N}(s,k) = \frac{1}{s}{k \choose s}{k-1 \choose s}
\end{equation}
are Narayana numbers. Stirling's approximation gives \cite{soshnikov2002note}
\begin{equation*}
\sum_{k=1}^{s} \gamma^k \bm{N}(s_N,k) = \frac{\gamma^{1/4}}{2\sqrt{\pi}}\frac{N(1+\sqrt{\gamma})^{2s+1}}{s^{3/2}}(1 + o(1)).
\end{equation*}

Marchenko and Pastur proved in \cite{marchenko1967distribution} that the spectrum of the product $\X_N^\top\X_N$ converges almost surely weakly to the limiting distribution (\ref{eq:MP_pdf_def}). Later this result was strengthened in \cite{johnstone2001distribution} and other works.

P{\'e}ch{\'e} proved \cite{peche2009universality} the universality of the joint distribution of top eigenvalues of SCM for a rich family of marginal distributions by developing a tight bound on the expected high-order moments. Adapted to our setup their main technical result reads as follows.
\begin{lemma}[Corollary from Propositions 2.4 and 2.5 from \cite{peche2009universality}]
\label{lem:mom_conv_norm_or}
Let $\X_N \in \mathcal{X}_{N,p_N}$ and $s_N = o\(\sqrt{N}\)$, then
\begin{multline}
\mathbb{E} \[\Tr{\(\frac{\X_N^\top\X_N}{(1+\sqrt{\gamma})^2}\)^{s_N}}\] \\ = \frac{\gamma^{1/4}(1+\sqrt{\gamma})}{2\sqrt{\pi}}\frac{N}{s_N^{3/2}}(1 + o(1)),
\end{multline}
as $N \to +\infty$, and the random variables
\begin{equation*}
\Tr{\(\frac{\X_N^\top\X_N}{(1+\sqrt{\gamma})^2}\)^{s_N}} - \mathbb{E} \[\Tr{\(\frac{\X_N^\top\X_N}{(1+\sqrt{\gamma})^2}\)^{s_N}}\]
\end{equation*}
converge in distribution to the normal law $\mathcal{N}\(0,\frac{1}{\pi}\)$.
\end{lemma}

Below we utilize the following result from \cite{peche2009universality}.

\begin{lemma}[Corollary of Theorem 3.1 from \cite{peche2009universality}]
Let $\X_N \in \mathcal{X}_{N,p_N}$, then for any sequence $s_N = O\(N^{2/3}\)$,
\begin{equation*}
\mathbb{E}\[\Tr{\(\frac{\X_N^\top\X_N}{(1+\sqrt{\gamma})^2}\)^{s_N}}\] \leqslant c_{MP}(\gamma,s_N)N,
\end{equation*}
where $c_{MP}(\gamma,s_N)$ is bounded uniformly over $N$.
\end{lemma}

\section{Pseudo-Random Ensembles}
\label{sec:pseudo-rand}
\subsection{Definitions}
In this section, we recall some definitions from \cite{soloveychik2017pseudo} and introduce a family of pseudo-$\,$Marchenko-Pastur (pseudo-MP) ensembles analogous to the pseudo-Wigner matrices.

\begin{definition}[\cite{soloveychik2017pseudo}]
\label{def:ps_wig}
Let $\x = \{X_i\}_{i=1}^N$ be a sequence of sign-valued random variables. $\x$ is $r$-independent if any $r$ of its elements $X_{i_1},\dots,X_{i_r}$ are statistically independent,
\begin{equation}
\mathbb{P}\[X_{i_1}=b_1,\dots,X_{i_r}=b_r\] = \prod_{l=1}^r \mathbb{P}\[X_{i_l}=b_l\],
\end{equation}
for any $i_1\neq\dots\neq i_r$ in the range $[1,N]$ and $b_i \in \{\mypm 1\}$.
\end{definition}

\begin{definition}[\cite{soloveychik2017pseudo}]
\label{def:ps_wig}
Let a subset $\mathcal{A}_N^r \subset S_N$ be endowed with the uniform measure. We say that it is an $r$-independent pseudo-Wigner ensemble of order $N$ if the elements of the upper triangular (including the main diagonal) parts of its matrices form an $r$-independent sequence w.r.t. (with respect to) the measure induced on them by $\mathcal{A}_N^r$.
\end{definition}

\begin{definition}[$r$-independent Pseudo-MP Ensemble of order $N$]
\label{def:ps_wish}
Let a subset $\mathcal{Y}_{N,p}^r \subset M_{N,p}$ be endowed with the uniform measure. We say that the ensemble of matrices
\begin{equation}
\{\Y_N\Y_N^\top \mid \Y_N \in \mathcal{Y}_{N,p}^r\}
\end{equation}
is an \textbf{$r$-independent pseudo-MP ensemble of order $N$} if the elements of the matrices $\Y_N$ form an $r$-independent sequence w.r.t. the measure induced on them by $\mathcal{Y}_{N,p}^r$.
\end{definition}

Below, whenever probability measure over $\mathcal{A}_N^r$ or $\mathcal{Y}_{N,p}^r$ are considered, they are always assumed to be uniform as in Definitions \ref{def:ps_wig} and \ref{def:ps_wish}.

The last definition is justified by the following result.
\begin{prop}
\label{th:main_res}
Let $q < e$, then for $r \leqslant q\, \log_2 N$ and any $\alpha \in (\frac{q}{e},1)$ there exists $N_0$ such that for any $N \geqslant N_0$, with probability at least $1-\frac{r}{N^{2(1-\alpha)}}$ a matrix $\Y_N$ chosen uniformly from $M_{N,p_N}^{2r}$ satisfies
\begin{equation}
\label{eq:main_bound}
\left| F_{\Y_N^\top\Y_N}(x) - F_{MP}(x)\right| \leqslant \frac{1}{r},\quad \forall x \in \mathbb{R}.
\end{equation}
\end{prop}
\begin{proof}
The proof from \cite{soloveychik2017pseudo} applies with minor changes.
\end{proof}

\subsection{High-Order Moments}
\begin{lemma}
\label{lem:mom_conv_norm_wig}
Let $\{\beta_N\},\; \{r_N\},\; \{s_N\} \subset \N$ be such that $s_N = o\(N^{2/3}\)$ with $s_N \leqslant \beta_N r_N$, and $\A_N$ be chosen uniformly from $\mathcal{A}_N^{\beta_N r_N}$, then for the expected moments we have
\begin{equation}
\mathbb{E} \[\Tr{\A_N^{s_N}}\] = \begin{cases} \sqrt{\frac{8}{\pi s_N^3}}N(1 + o(1)), & s_N \text{ even}, \\ 0, & s_N \text{ odd},\end{cases}
\end{equation}
as $N \rightarrow \infty$. In addition, the first $p=1,\dots,2\beta_N$ moments of the random variable
\begin{equation}
\Tr{\A_N^{s_N}} - \mathbb{E} \[\Tr{\A_N^{s_N}}\]
\end{equation}
converge to the moments of the normal law $\mathcal{N}\(0,\frac{1}{\pi}\)$.
\end{lemma}
\begin{proof}
The proof follows that of Main Theorem of \cite{sinai1998central}.
\end{proof}

\begin{lemma}
\label{lem:main_res_wig}
Let $\{r_N\},\; \{s_N\} \subset \N$ be such that $s_N = O\(N^{2/3}\)$ with $s_N \leqslant r_N$, and $\A_N$ be chosen uniformly from $\mathcal{A}_N^{r_N}$, then
\begin{equation}
\mathbb{E}\[\Tr{\A_N^{s_N}}\] \leqslant c_W(s_N)N,
\end{equation}
where $c_W(s_N)$ is bounded uniformly over $N$.
\end{lemma}
\begin{proof}
The proof is analogous to that of Theorem 2 from \cite{soshnikov1999universality}.
\end{proof}

\begin{lemma}
\label{lem:mom_conv_norm_wish}
Let $\{\beta_N\},\; \{r_N\},\; \{s_N\} \subset \N$ be such that $s_N = o\(\sqrt{N}\)$ with $s_N \leqslant \beta_N r_N$, and $\Y_N$ be chosen uniformly from $\mathcal{Y}_{N,p_N}^{\beta_N r_N}$, then
\begin{multline}
\mathbb{E} \[\Tr{\(\frac{\Y_N^\top\Y_N}{(1+\sqrt{\gamma})^2}\)^{s_N}}\] \\ = \frac{\gamma^{1/4}(1+\sqrt{\gamma})}{2\sqrt{\pi}}\frac{N}{s_N^{3/2}}(1 + o(1)),
\end{multline}
as $N \rightarrow \infty$. In addition, the first $p=1,\dots,2\beta_N$ moments of the random variable
\begin{equation*}
\Tr{\(\frac{\Y_N^\top\Y_N}{(1+\sqrt{\gamma})^2}\)^{s_N}} - \mathbb{E} \[\Tr{\(\frac{\Y_N^\top\Y_N}{(1+\sqrt{\gamma})^2}\)^{s_N}}\]
\end{equation*}
converge to the moments of the normal law $\mathcal{N}\(0,\frac{1}{\pi}\)$.
\end{lemma}
\begin{proof}
The proof is analogous to those of Propositions 2.4 and 2.5 from \cite{peche2009universality}.
\end{proof}

\begin{lemma}
\label{th:main_res_wish}
Let $\{r_N\},\; \{s_N\} \subset \N$ be such that $s_N = O\(N^{2/3}\)$ with $s_N \leqslant r_N$, and $\Y_N$ be chosen uniformly from $\mathcal{Y}_{N,p_N}^{r_N}$, then
\begin{equation*}
\mathbb{E}\[\Tr{\(\frac{\Y_N^\top\Y_N}{N(1+\sqrt{\gamma})^2}\)^{s_N}}\] \leqslant c_{MP}(\gamma,s_N)N,
\end{equation*}
where $c_{MP}(\gamma,s_N)$ is bounded uniformly over $N$.
\end{lemma}
\begin{proof}
The proof is analogous to that of Theorem 3.1 from \cite{peche2009universality}.
\end{proof}

\section{Spectral Norms}
\label{sec:spect_norm}
Here we present the main results of the article.
\subsection{Pseudo-Wigner Matrices}
\begin{prop}
\label{prop:main_res_wig}
Let $\A_n \in \mathcal{A}_N^{r_N}$ with $\liminf \frac{r_N}{N^\rho} > 0$ for some $\rho \in (0,1]$, then for any $\varepsilon > 0$
\begin{equation}
\|\A_N\| = 1 + o\(\frac{\log^{1+\varepsilon} N}{N^{\min[\rho,2/3]}}\), \quad \text{a.s.}
\end{equation}
\end{prop}
\begin{proof}
For simplicity, let us start with the case $\rho \leqslant \frac{2}{3}$.
Given $\varepsilon > 0$, set
\begin{equation}
q_N = 2 \floor*{\frac{1}{2}\frac{N^\rho}{\log^{\varepsilon/2}N}}.
\end{equation}
Using Markov's inequality we obtain the following chain of bounds,
\begin{align}
\label{eq:prob_chain}
\mathbb{P}&\bigg\{\|\A_N\| \geqslant 1 + \frac{\log^{1+\varepsilon}N}{N^\rho}\bigg\} \\
& \leqslant \mathbb{P}\bigg\{\Tr{\A_N^{q_N}} \geqslant \(1 + \frac{\log^{1+\varepsilon}N}{N^\rho}\)^{q_N}\bigg\} \nonumber \\
& = \mathbb{P}\bigg\{\Tr{\A_N^{q_N}} \geqslant \(1 + \frac{\log^{1+\varepsilon}N}{N^\rho}\)^{2\[\frac{1}{2}\frac{N^\rho}{\log^{\varepsilon/2}N}\]}\bigg\} \nonumber \\
& \leqslant \mathbb{P}\bigg\{\Tr{\A_N^{q_N}} \geqslant \frac{1}{2}\exp\(\log^{1+\varepsilon/2} N\)\bigg\} \nonumber \\
& \leqslant \frac{\mathbb{E}\[\Tr{\A_N^{q_N}}\]}{\frac{1}{2}\exp\(\log^{1+\varepsilon/2} N\)} = O\(N \exp\(-\log^{1+\varepsilon/2} N\)\), \nonumber
\end{align}
where the last line follows from Lemma \ref{lem:main_res_wig}. This implies
\begin{equation}
\sum_{N=1}^\infty \mathbb{P}\bigg\{\|\A_N\| \geqslant 1 + \frac{\log^{1+\varepsilon}N}{N^\rho}\bigg\} < + \infty.
\end{equation}
It now follows from Borel-Cantelli lemma that
\begin{equation}
\label{eq:one_dir_BC}
\|\A_N\| \leqslant 1 + \frac{\log^{1+\varepsilon}N}{N^\rho},\quad \text{a.s.}
\end{equation}
In order to get the opposite direction inequality, note that Lemma \ref{lem:mom_conv_norm_wig} together with the linear algebraic relation
\begin{equation}
\|\A_N\| \leqslant \Tr{\A_N^q}^{1/q} \leqslant N^{1/q}\|\A_N\|,
\end{equation}
give
\begin{equation}
\mathbb{E}\[\|\A_N\|\] = 1 + o\(\frac{1}{N^\kappa}\),
\end{equation}
for any fixed positive $\kappa$ and therefore,
\begin{equation}
\|\A_N\| \geqslant 1 + \frac{\log^{1+\varepsilon}N}{N^\rho},\quad \text{a.s.}
\end{equation}
which together with (\ref{eq:one_dir_BC}) implies the desired statement.

Assume now that $\rho > \frac{2}{3}$. We know from \cite{soshnikov1999universality} that Lemma \ref{lem:mom_conv_norm_or} is no longer valid in this case and the expected traces can grow faster that $O(N)$. Therefore, to keep the first ratio in the last line of (\ref{eq:prob_chain}) bounded by a summable sequence, the largest (in order) possible choice for $q_N$ is
\begin{equation}
q_N = 2 \floor*{\frac{1}{2}\frac{N^{2/3}}{\log^{\varepsilon/2}N}}.
\end{equation}
Now the same reasoning as above together with Lemma \ref{lem:main_res_wig} complete the proof.
\end{proof}


\subsection{Pseudo-Wishart Matrices}
\begin{prop}
Let $\Y_N$ be chosen uniformly from $\mathcal{Y}_{N,p_N}^{r_N}$ for some $\rho \in (0,1]$, then for any $\varepsilon > 0$
\begin{equation}
\norm{\frac{\Y_N^\top\Y_N}{(1+\sqrt{\gamma})^2}} = 1 + o\(\frac{\log^{1+\varepsilon} N}{N^{\min[\rho,2/3]}}\), \quad \text{a.s.}
\end{equation}
\end{prop}
\begin{proof}
The proof of Proposition \ref{prop:main_res_wig} works verbatim with Lemmas \ref{lem:mom_conv_norm_wish} and \ref{th:main_res_wish} replacing Lemmas \ref{lem:mom_conv_norm_wig} and \ref{lem:main_res_wig}, correspondingly.
\end{proof}


\section{A Construction from Dual BCH codes}
Next we provide an explicit constructions of the pseudo-Wigner and pseudo-MP ensembles from dual BCH codes. The idea was presented in \cite{soloveychik2017pseudo} for the $r$-independent pseudo-Wigner matrices with $r$ of the order of $\log_2 N$. Here we focus on higher levels of independence with $r \propto N^\rho,\; \rho > 0$.

For $m \in \N$, a primitive narrow-sense binary BCH code $\mathcal{C}_m^\delta$ of length $n=2^m-1$ and designed minimum distance $\delta \geqslant 3$ is a cyclic code generated by the lowest degree binary polynomial having roots $\alpha,\;\alpha^2,\dots,\alpha^{\delta-1}$, where $\alpha$ is a primitive element of $GF(2^m)$. 

\begin{lemma}[Theorem 9.1.1, Theorem 9.2.6 from \cite{macwilliams1977theory}]
\label{th:bch_dim}
A primitive narrow-sense binary BCH code $\mathcal{C}_m^\delta$ of length $n=2^m-1$ and designed distance $\delta$ has
\begin{itemize}
\item minimum distance $d$ such that $\delta \leqslant d \leqslant 2\delta-1$, and
\item dimension at least $n-mt$.
\end{itemize}
\end{lemma}

Under the same assumptions as in Lemma \ref{th:bch_dim}, the dual BCH code is a cyclic code of dimension $k^\perp \leqslant mt$  \cite{macwilliams1977theory}.

\begin{lemma}[Lemma 3.2 from \cite{babadi2011spectral}]
\label{lem:indep_source}
If a code $\mathcal{C}$ has minimum distance $d$, then its dual code $\mathcal{C}^\perp$ is $(d-1)$-independent (see Definition \ref{def:ps_wig}) w.r.t. to the uniform measure over its codewords. 
\end{lemma}

Given these results, the pseudo-Wigner matrices are built as explained in Section as IV of \cite{soloveychik2017pseudo}. Pseudo-MP matrices are constructed analogously, by first packing the codewords of the dual BCH code row by row into  rectangular $N \times p$ matrices $\Y_N$ scaled by $\frac{1}{\sqrt{N}}$. Then the desired SCMs are obtained as $\Y_N^\top\Y_N$.

\section{Numerical Simulations}
\label{sec:num}
To illustrate the results obtained in Section \ref{sec:spect_norm}, we constructed a BCH code of length $n=2^{14}-1=16383$ and minimum distance $15$ (the generating polynomial was computed by calling \texttt{bchgenpoly(16383,16173)} function of \textsc{Matlab}). Using the obtained polynomial, we calculated the generating polynomial of the dual code as explained in \cite{soloveychik2017pseudo} and randomly chose $10^5$ words from the dual code. These codewords were packed into $180 \times 180$ symmetric sign matrices as described in Section as IV of \cite{soloveychik2017pseudo}. In Figure \ref{fig} the empirical distribution of the spectral norms of the obtained pseudo-Wigner matrices (dBCH curve in the picture) is compared to the theoretical limit for the truly random matrices, the so-called Tracy-Widom distribution \cite{tracy1994level}.

\begin{figure}[!t]
\hspace{-0.5cm}
\includegraphics[width=3.6in]{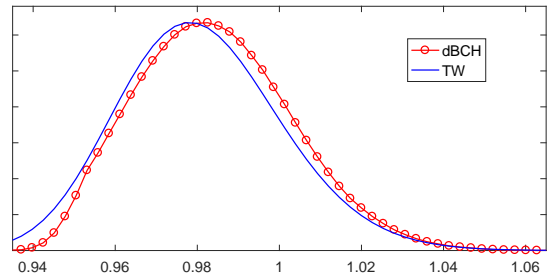}
\vspace{-0.5cm}
\caption{Distribution of norms of pseudo-Wigner matrices constructed from a dual BCH code, $N = 180,\; m=14,\; d=15$ versus Tracy-Widom law.}
\label{fig}
\end{figure}

\section{Conclusions}
\label{sec:conc}
In this article, we extend the framework of pseudo-Wigner matrices introduced in \cite{soloveychik2017pseudo} to a new family of pseudo-$\,$Marchenko-Pastur ensembles. The definitions of both classes of matrices are based on the concept of $r$-independence of the matrix entries to mimic the behavior of the truly random Wigner and sample covariance ensembles, correspondingly. The designed properties of these pseudo-random ensembles allow us to derive approximations of the expected moments similar to those for corresponding truly random matrices, which further enables us to achieve bounds on the spectral norms of the pseudo-Wigner and pseudo-MP ensembles as functions of the level of independence $r$. We also provide explicit constructions of pseudo-Wigner and pseudo-MP ensembles from dual BCH codes. 

\section{Acknowledgment}
This work was supported by the Fulbright Foundation and Army Research Office grant No. W911NF-15-1-0479.

\bibliographystyle{IEEEtran}
\bibliography{ilya_bib}
\end{document}